\documentclass{article}

\usepackage{amsmath}
\usepackage{graphicx}
\usepackage{subcaption}
\usepackage[top=25mm, bottom=25mm, left=23.5mm, right=23.5mm]{geometry}
\usepackage{amsfonts}
\usepackage{mathtools}
\usepackage{amssymb}
\usepackage{amsthm}
\usepackage{relsize}
\usepackage[cal=rsfso,calscaled=1.1]{mathalfa}
\usepackage{hyperref}

\hypersetup{
	colorlinks=true,
	linkcolor=black,
	citecolor=black,
	filecolor=black,
	urlcolor=black,
}

\theoremstyle{definition}
\newtheorem{proposition}{Proposition}
\newtheorem{lemma}[proposition]{Lemma}
\newtheorem{theorem}[proposition]{Theorem}
\newtheorem{corollary}[proposition]{Corollary}
\usepackage{amsthm}

\makeatletter
\renewenvironment{proof}[1][\proofname] {\par\pushQED{\qed}\normalfont\topsep6\p@\@plus6\p@\relax\trivlist\item[\hskip\labelsep\bfseries#1\@addpunct{.}]\ignorespaces}{\popQED\endtrivlist\@endpefalse}
\makeatother

\title{A New Upper Bound for Cancellative Pairs}
\author{Barnab\'{a}s Janzer\thanks{Trinity College, Cambridge CB2 1TQ, United Kingdom. Email: bkj21@cam.ac.uk}}
\date{\vspace{-21pt}}

\begin{document}
	\maketitle

\begin{abstract}
A pair $(\mathcal{A},\mathcal{B})$ of families of subsets of an $n$-element set is called cancellative if whenever $A,A'\in\mathcal{A}$ and $B\in\mathcal{B}$ satisfy $A\cup B=A'\cup B$, then $A=A'$, and whenever $A\in\mathcal{A}$ and $B,B'\in\mathcal{B}$ satisfy $A\cup B=A\cup B'$, then $B=B'$. It is known that there exist cancellative pairs with $|\mathcal{A}||\mathcal{B}|$ about $2.25^n$, whereas the best known upper bound on this quantity is $2.3264^n$. In this paper we improve this upper bound to $2.2682^n$. Our result also improves the best known upper bound for Simonyi's sandglass conjecture for set systems.
\end{abstract}\vspace{-5pt}

\section{Introduction}

The notion of a cancellative pair was introduced by Holzman and K\"{o}rner \cite{holzmankorner}. We say that a pair $(\mathcal{A},\mathcal{B})$ of families of subsets of an $n$-element set $S$ is \textit{cancellative} if
\begin{align}
\begin{split}
&\textnormal{whenever $A, A'\in\mathcal{A}$ and $B\in\mathcal{B}$ satisfy $A\cup B=A'\cup B$ then $A=A'$}\\
\textnormal{and }&\textnormal{whenever $A\in\mathcal{A}$ and $B, B'\in\mathcal{B}$ satisfy $A\cup B=A\cup B'$ then $B=B'$;}
\end{split}
\end{align}
or, equivalently,
\begin{align}\label{differencedef}
\begin{split}
&\textnormal{whenever $A, A'\in\mathcal{A}$ and $B\in\mathcal{B}$ satisfy $A\setminus B=A'\setminus B$ then $A=A'$}\\
\textnormal{and }&\textnormal{whenever $A\in\mathcal{A}$ and $B, B'\in\mathcal{B}$ satisfy $B\setminus A=B'\setminus A$ then $B=B'$.}
\end{split}
\end{align}
We will usually take $S=[n]=\{1, ..., n\}$ and will call a cancellative pair with $\mathcal{A}=\mathcal{B}$ a \textit{symmetric cancellative pair}. Note that the assumption that $(\mathcal{A},\mathcal{A})$ is a symmetric cancellative pair is slightly stronger than the assumption that $\mathcal{A}$ is a \textit{cancellative family}, meaning no three distinct sets $A, B, C\in\mathcal{A}$ satisfy $A\cup B=A\cup C$ \cite{franklfuredi}. We mention that the concept of cancellative pairs corresponds to the information theoretic concept of uniquely decodable code pairs for the binary multiplying channel without feedback (see e.g. Tolhuizen \cite{tolhuizen}).\bigskip

In the case when $n$ is a multiple of $3$, we can obtain an example of a  symmetric cancellative pair the following way. Partition $S$ into $n/3$ classes of size $3$, and take $\mathcal{A}$ (and $\mathcal{B}$) to be the collection of subsets of $S$ containing exactly one element from each class. It is not hard to verify that we get a cancellative pair. Here we have $|\mathcal{A}||\mathcal{B}|=3^{2n/3}$, where $3^{2/3}\approx 2.08$. In the symmetric case, Erd\H{o}s and Katona \cite{katona} conjectured this to be the maximal size for cancellative families. A counterexample was found by Shearer \cite{shearer}.
Tolhuizen \cite{tolhuizen} gave a beautiful construction to show that we can achieve $(|\mathcal{A}||\mathcal{B}|)^{1/n}\to 9/4=2.25$, even by symmetric pairs. This construction is (asymptotically) optimal in the symmetric case by a result of Frankl and F\"{u}redi \cite{franklfuredi}.\bigskip

In the general (non-symmetric) case, the exact value of $\alpha=\sup (|\mathcal{A}||\mathcal{B}|)^{1/n}$ is not known. The best known upper bound is due to Holzman and K\"{o}rner \cite{holzmankorner}, who showed that $|\mathcal{A}||\mathcal{B}|<\theta^n$ where $\theta\approx 2.3264$. No lower bound better than Tolhuizen's (symmetric) $2.25$ is known. Our main aim in this paper is to improve the upper bound to $2.2682^n$. Our proof requires some numerical calculations by a computer.\bigskip

A related concept is that of a recovering pair. A pair $(\mathcal{A},\mathcal{B})$ of collections of subsets of an $n$-element set $S$ is called \textit{recovering} \cite{simonyi, holzmankorner} if for all $A, A'\in\mathcal{A}$ and $B, B'\in\mathcal{B}$ we have \vspace{-1pt}
\begin{equation}
A\setminus B=A'\setminus B'\implies A=A'\textnormal{\hspace{12pt} and \hspace{12pt}}B\setminus A=B'\setminus A'\implies B=B'.
\end{equation}
So any recovering pair is cancellative (cf. (\ref{differencedef})). Simonyi's sandglass conjecture for set systems \cite{simonyi} states that $|\mathcal{A}||\mathcal{B}|\leq 2^n$ for a recovering pair. (The value of $2^n$ may be obtained by taking $\mathcal{A}=\mathcal{P}(S_1)$, $\mathcal{B}=\mathcal{P}(S\setminus S_1)$ for any $S_1\subseteq S$. There is a more general sandglass conjecture for lattices, due to Ahlswede and Simonyi \cite{simonyi}.)  Our upper bound of $2.2682^n$ is an improvement on the previous best known bound of $2.284^n$ (Solt\'{e}sz, \cite{soltesz}).

\section{Proof of the upper bound}
Let $h(x)=-x\log_2 x-(1-x)\log_2{(1-x)}$ be the binary entropy function (with the convention $0\log_2 0=0$).\\
Define $\mathcal{A}_{i}=\{A\in\mathcal{A}\mid i\notin A\}$ and $p_i=|\mathcal{A}_{i}|/|\mathcal{A}|$; $q_i$ is defined similarly for $\mathcal{B}$. We quote the following result of Holzman and K\"{o}rner \cite{holzmankorner}. (We will ignore the case when $\mathcal{A}$ or $\mathcal{B}$ is empty.)

\begin{proposition}[Holzman and K\"{o}rner \cite{holzmankorner}] For a cancellative pair $(\mathcal{A},\mathcal{B})$, we have
\begin{equation}
\log_2\left[|\mathcal{A}||\mathcal{B}|\right]\leq {\sum_{i=1}^{n}f(p_i,q_i)} \label{entropy}
\end{equation}
where $f(p,q)=ph(q)+qh(p)$.

\end{proposition}
The result above can be established by considering the entropies of each of the random variables of the form $\xi^B=A\setminus B$, where $B\in\mathcal{B}$ is fixed and $A\in\mathcal{A}$ is chosen uniformly at random (and doing the same with $\mathcal{A}$, $\mathcal{B}$ interchanged). Holzman and K\"{o}rner \cite{holzmankorner} used (\ref{entropy}) and induction to establish their upper bound of $|\mathcal{A}||\mathcal{B}|<\theta^n$ ($\theta\approx 2.3264$).\bigskip

However, this argument can be improved. We call a cancellative pair \textit{$k$-uniform} if $|A|=|B|=k$ for all $A\in\mathcal{A}$, $B\in\mathcal{B}$. As we will see, bounding $|\mathcal{A}||\mathcal{B}|$ for $k$-uniform families enables us to give bounds for general (non-uniform) pairs. For $n/k$ small, it is easy to give efficient bounds, and for $n/k$ large, we will use that the growth speed of the maximum of $|\mathcal{A}||\mathcal{B}|$ (with $k$ fixed, $n$ increasing) can be bounded.\bigskip

If $(\mathcal{A},\mathcal{B})$ and $(\mathcal{A'},\mathcal{B'})$ are cancellative pairs over disjoint ground sets $S$ and $S'$, define their product $(\mathcal{A''},\mathcal{B''})$ by
$$\mathcal{A''}=\{A\cup A'\mid A\in \mathcal{A}, A'\in\mathcal{A'}\}$$
$$\mathcal{B''}=\{B\cup B'\mid B\in \mathcal{B}, B'\in\mathcal{B'}\}$$
giving a cancellative pair over $S\cup S'$ with $|\mathcal{A''}||\mathcal{B''}|=|\mathcal{A}||\mathcal{B}||\mathcal{A'}||\mathcal{B'}|$.\\
(Note that the cancellative pair in the Introduction is just the product of cancellative pairs of the form $n=3$,\\ $\mathcal{A}=\mathcal{B}=\{\{1\},\{2\},\{3\}\}$.) Let $c(n)$ be the maximum of $|\mathcal{A}||\mathcal{B}|$ for a cancellative pair over an $n$-element set, and let $c_k(n)$ be the maximum considering only $k$-uniform pairs. Similarly to \cite{soltesz}, we prove the following lemma.
\begin{lemma}\label{product}
Let $M$ be a fixed positive integer, and suppose that $\beta>0$ is such that $c_k(n)\leq \beta^n$ for all $k$ divisible by $M$ and for all $n\geq k$. Then $c(n)\leq \beta^n$ for all $n$.	
\end{lemma}
\begin{proof}
	Suppose the conditions above are satisfied but $|\mathcal{A}||\mathcal{B}|=\omega^n$ for some $\omega>\beta$. Take the product of $(\mathcal{A},\mathcal{B})$ with (a copy of) $(\mathcal{B},\mathcal{A})$ to get a cancellative pair $\left( \mathcal{A}_{(1)},\mathcal{B}_{(1)}\right) $ over some set $S$ with\\$\left|\mathcal{A}_{(1)}\right|=\left|\mathcal{B}_{(1)}\right|=\omega^{|S|/2}$ and $\mathcal{A}_{(1)}$ and $\mathcal{B}_{(1)}$ containing the same number of sets of size $t$ for any $t$. Also, we can take the product of $\left( \mathcal{A}_{(1)},\mathcal{B}_{(1)}\right)$ with (copies of) itself several times to get a pair with similar properties, so we may assume that $|S|$ is large enough so that $\omega^{|S|}/(|S|+1)^2>\beta^{|S|}$. Take $k_0\in \{0, 1, ..., |S|\}$ such that $\mathcal{A}_{(1)}, \mathcal{B}_{(1)}$ each contain at least $\omega^{|S|/2}/(|S|+1)$ sets of size $k_0$, let $\left(\mathcal{A}_{(2)},\mathcal{B}_{(2)}\right)$ contain only these $k_0$-sets. So $\left|\mathcal{A}_{(2)}\right|\left|\mathcal{B}_{(2)}\right|>\beta^{|S|}$ and $\left(\mathcal{A}_{(2)},\mathcal{B}_{(2)}\right)$ is $k_0$-uniform cancellative. Take the product of $\left(\mathcal{A}_{(2)},\mathcal{B}_{(2)}\right)$ with itself several times to obtain $\left(\mathcal{A}_{(2)}^M,\mathcal{B}_{(2)}^M\right)$, an $(Mk_0)$-uniform cancellative family contradicting our assumptions.
\end{proof}

We also need a simple observation.

\begin{lemma}\label{smallbound}
	If $k$ and $n\geq k$ are positive integers, then
	$c_k(n)\leq 2^{2(n-k)}$. In particular, $c_k(n)\leq 2^n$ for $n\leq 2k$.
\end{lemma}
\begin{proof}
	Given $A\in\mathcal{A}$, all $B\in\mathcal{B}$ have to differ on the complement of $A$, hence $|\mathcal{B}|\leq 2^{n-k}$. Similarly $|\mathcal{A}|\leq 2^{n-k}$.
\end{proof}

We note that we have equality for $k\leq n\leq 2k$ (i.e. $c_k(n)=2^{2(n-k)}$), even in the symmetric case \cite{franklfuredi}. Also, we could deduce Lemma \ref{smallbound} from (\ref{entropy}), observing that $\sum p_i=\sum q_i=n-k$.\bigskip

In order to state our key proposition, we need a definition. For $\gamma,x\geq 2$, consider the following optimisation problem:
\begin{equation}
\label{optimisation}
\begin{split}
\textnormal{maximize\hspace{12pt}} &\frac{1}{n}\sum_{i=1}^{n} f(p_i,q_i)\\
\textnormal{subject to\hspace{12pt}} &p_iq_i\leq 1/\gamma\hspace{12pt}\textnormal{for } i=1, ..., n\\
&\sum_{i=1}^n p_i=\sum_{i=1}^n q_i \geq n(1-1/x)\\
&0\leq p_i, q_i\leq 1\hspace{12pt}\textnormal{for } i=1, ..., n\\
&n\in\mathbb{N}
\end{split}
\end{equation}
(Note that the positive integer $n$ is not fixed.) We write $\varphi(\gamma,x)$ for the solution (i.e. the supremum) of (\ref{optimisation}).

\begin{proposition} \label{inductivebound}
	Suppose $k$ is a positive integer, $2\leq\lambda$ such that $\lambda k$ is an integer, and $2\leq r_1\leq \gamma$. Suppose that $c_k(\lambda k)\leq r_1^{\lambda k}$ and
\begin{equation}\label{inductiveineq}
	r_{1}\geq 2^{\varphi(\gamma,\lambda)}.
\end{equation}
Then, for $\lambda k\leq n$,
$$c_k(n)\leq r_1^{\lambda k}\gamma^{n-\lambda k}.$$
In particular, if $\mu>\lambda$, $\mu k$ is an integer and $r_2=r_1^{\lambda/\mu}\gamma^{1-\lambda/\mu}$, then $c_k(n)\leq r_{2}^n$ for $\lambda k\leq n\leq \mu k$.
\end{proposition}

\begin{proof}
	Notice that $\gamma\geq r_2\geq r_1$.
We know the given inequality holds for $n=\lambda k$. Suppose it is false for some $n$, $\lambda k+1\leq n$, $n$ minimal.\\
Then $c_k(n)/c_k(n-1)>\gamma$. So we must have $p_iq_i<1/\gamma$ (or else $|\mathcal{A}_{i}||\mathcal{B}_{i}|>c_k(n-1)$ and $(\mathcal{A}_i,\mathcal{B}_i)$ is cancellative).\\
We also have $\sum p_i=\sum q_i=n-k=n(1-k/n)\geq n(1-1/\lambda)$. Hence $\sum f(p_i,q_i)\leq n\varphi(\gamma,\lambda)$ (by the definition of $\varphi$).
So then, by (\ref{entropy}), we get
$$|\mathcal{A}||\mathcal{B}|\leq 2^{n\varphi(\gamma,\lambda)}\leq r_{1}^n\leq r_1^{\lambda k}\gamma^{n-\lambda k},$$
contradiction.\\
For $\lambda k\leq n\leq \mu k$, we have $c_k(n)^{1/n}\leq (r_1/\gamma)^{\lambda k/n}\gamma\leq (r_1/\gamma)^{\lambda/\mu}\gamma=r_2$.
\end{proof}

Proposition \ref{inductivebound} enables us to implement the following method.
Let $2=\lambda_0<\lambda_1<...<\lambda_N$, and let $\rho_0=2$.\\
Using a computer program, we find some $\rho_1\geq\rho_0$, then $\rho_2\geq\rho_1$, and so on, finally $\rho_N$, such that the conditions of Proposition \ref{inductivebound} hold for $\lambda=\lambda_i$, $\mu=\lambda_{i+1}$, $r_1=\rho_i$, $r_2=\rho_{i+1}$ ($i=0, 1, ..., N-1$). So then $c_k(n)\leq\rho_N^n$ for $n/k\leq \lambda_N$.\medskip

To be able to apply this method, we make the following observations.\vspace{-3pt}
\begin{enumerate}
	\item If $\lambda_i$ is rational for all $i$, then we are allowed to assume that $\lambda_i k$ is an integer (since we may assume $M$ divides $k$ for any fixed $M$ positive integer).\vspace{-6pt}
	\item We do not need to consider $n/k>3.6$. Indeed, for $n/k>3.6$ we have $p_i+q_i>2(1-1/3.6)=13/9$ for some $i$, so then $p_iq_i>1\cdot 4/9=1/2.25$. Hence $c_k(n)< 2.25c_k(n-1)$, as $(\mathcal{A}_i,\mathcal{B}_i)$ is cancellative.\vspace{-6pt}
	\item We need to find an upper bound on $\varphi(\gamma,x)$. Details on how this is done are given in the Appendix, however, we note the following simple result.\\
	Let $\gamma\geq 2.25$, $x\geq 2$ and let $(p_0,q_0)$ satisfy $p_0+q_0=2(1-1/x)$ and $p_0q_0=1/\gamma$.\\
	If $0\leq p_0, q_0\leq 1$, $p_0\not=q_0$, then $\varphi(\gamma,x)=f(p_0,q_0)$.
\end{enumerate}

Now we are ready to prove our result using the method described above. Choose, for example, $N=100000$ and $\lambda_i=2+i(3.6-2)/N$. Then find appropriate values of $\rho_1, ..., \rho_N$ using a computer program. Details about our implementation are given in the Appendix. Our program gives $\rho_N=2.268166...$, whence $c_k(n)\leq 2.2682^n$ for all $n$ (and $k$ a multiple of an appropriate $M$). By Lemma \ref{product}, we get our main result.
\begin{theorem}
	For a cancellative pair $(\mathcal{A},\mathcal{B})$ over an $n$-element set, we have $|\mathcal{A}||\mathcal{B}|\leq 2.2682^n$.\qed
\end{theorem}\vspace{3pt}
\section{Remarks}
\paragraph{Uniform constructions}
We now discuss how our upper bound on $c_k(n)$ is related to the best known $k$-uniform constructions as $n/k$ varies.
Tolhuizen \cite{tolhuizen} gave a family of symmetric $k$-uniform pairs for all values of $k$ and $n$ having
$|\mathcal{A}|\geq \nu \binom{n}{k}2^{-k}$, where $\nu$ is a constant. It follows that for $n/k=x>2$, we have $$c_k(n)^{1/n}\geq 2^{2(h(1/x)-1/x)+o(1)}.$$
This construction is known to be asymptotically optimal in the symmetric $k$-uniform case \cite{franklfuredi, tolhuizen}.\\(As pointed out after Lemma \ref{smallbound}, the exact value of $c_k(n)$ is known for $n/k\leq 2$.)\medskip

Figure \ref{graph} shows the upper bound we obtain by the argument above for $c_k(n)^{1/n}$, together with the lower bound from Tolhuizen's construction ($n/k$ fixed, $n$ large). We note that, with a slight modification of Proposition \ref{inductivebound}, our upper bound could be decreased for $n/k$ large (instead of becoming constant at the maximum value). However, this would not improve our constant of $2.2682$, and it requires more care to find bounds for the optimization problem (\ref{optimisation}) when $\gamma$ is small.\vspace{-2pt}
\begin{figure}[h!]
	\includegraphics[clip,trim=1.5cm 6.7cm 1cm 8.7cm, width=0.69\linewidth]{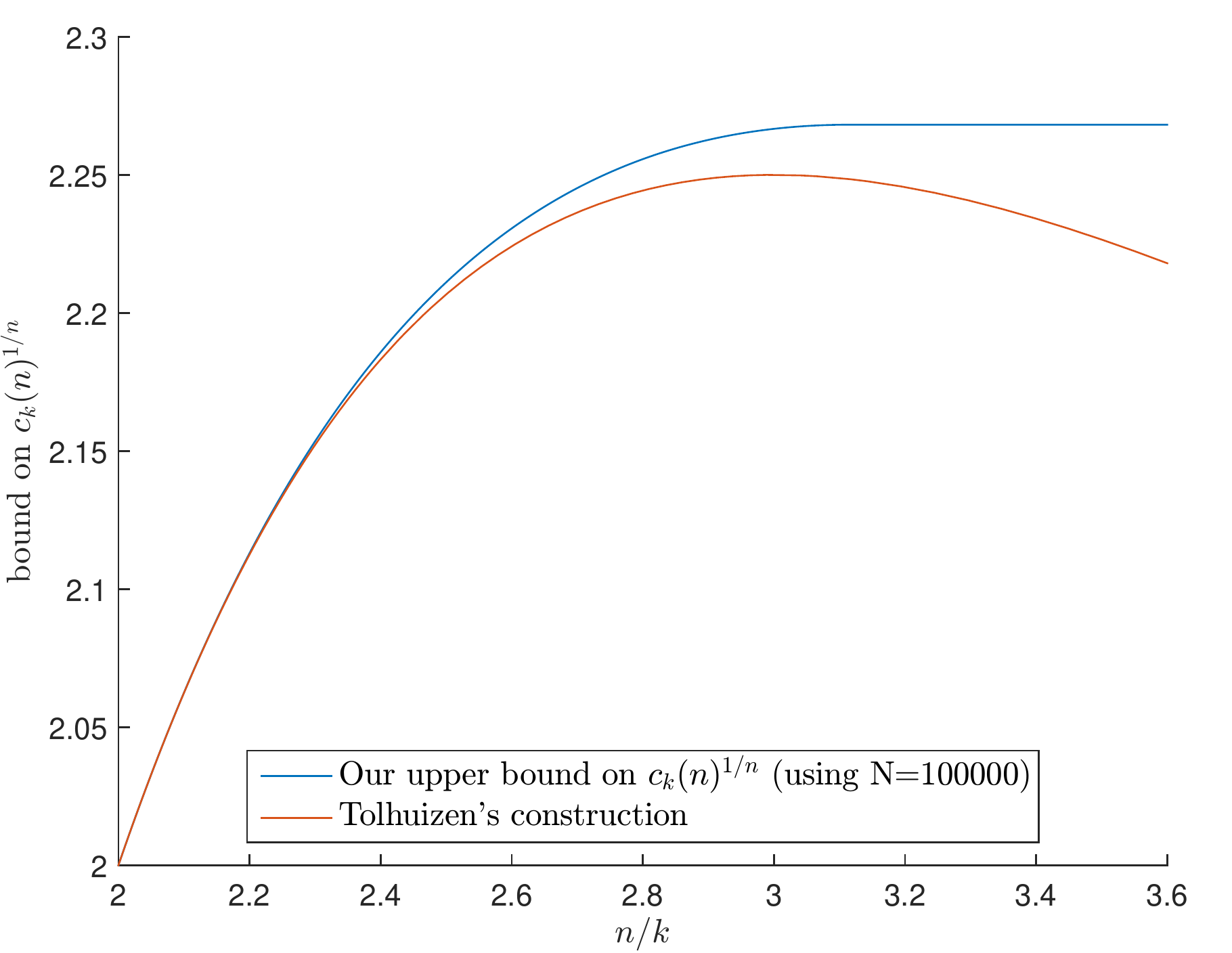}
	\centering
	\caption{Graphical representation of the lower and upper bounds for uniform pairs.}
	\label{graph}
\end{figure}

\paragraph{The symmetric case} In the case $\mathcal{A}=\mathcal{B}$, an argument similar to the one considered above gives the best possible bound of $2.25^n$. In fact, our argument is equivalent to that of Frankl and F\"uredi \cite{franklfuredi}. For convenience, we consider $G_k(n)$, the largest possible size of $\mathcal{A}$ if $(\mathcal{A},\mathcal{A})$ is $k$-uniform cancellative. (So then $c_k(n)\geq G_k(n)^2$.)\\
In this case, we have $p_i=q_i$ for each $i$. If $G_k(n)/G_k(n-1)=\gamma$, then $p_i\leq 1/\gamma$ for all $i$. But $\sum p_i=n-k$, hence $\gamma\leq\frac{n}{n-k}$. As $G_k(2k)\leq 2^{k}$, induction gives (for $n\geq 2k$) \vspace{-1pt} $$G_k(n)\leq 2^k{\binom{n}{k}}\bigg/{\binom{2k}{k}}$$\\
This is exactly the formula obtained by Frankl and F\"uredi \cite{franklfuredi}. This is not surprising: their argument is essentially the same, but instead of removing elements one-by-one (i.e. inducting from $n-1$ to $n$), they consider a random set of size $2k$. (It is not hard to deduce the bound $(3/2)^{2n}$ for symmetric pairs from here, noticing that subexponential factors can be ignored by a product argument. The asymptotic optimality of Tolhuizen's construction for $k$-uniform symmetric cancellative pairs ($n\to\infty$, $n/k\to x>2$) also follows \cite{tolhuizen}.)\bigskip

\paragraph{Recovering pairs}
Since any recovering pair is also cancellative, the result above immediately gives the following corollary.
\begin{corollary}
	For a recovering pair $(\mathcal{A},\mathcal{B})$ over an $n$-element set, we have $|\mathcal{A}||\mathcal{B}|\leq 2.2682^n$.\qed
\end{corollary}

We remark that a bound stronger than $2^{2k}$ for $k$-uniform recovering pairs over a $2k$-element set would give a stronger bound on the maximal value of $|\mathcal{A}||\mathcal{B}|$ using the argument above (we could choose $\rho_0$ to be smaller). Note that the product of recovering families is recovering \cite{soltesz}, so our arguments would still be valid.

\section*{Appendix}
The appendix contains two main parts. In the first part, we give bounds for $\varphi(\gamma,x)$. In the second part, we briefly describe how we implement our argument using a computer program.
\subsection*{Bounding the optimisation problem}

\begin{lemma} \label{whereistheoptimum}
	Suppose $\gamma\geq 2.25$ and $\kappa\geq 0$. Then the maximizer $(p,q)$ of $L_\kappa(p,q)=f(p,q)+\kappa (p+q)$ in the range $0\leq p,q\leq 1$, $pq\leq 1/\gamma$ satisfies $pq=1/\gamma$.
\end{lemma}
\begin{proof}
	Consider the maximizer. We may assume $p\leq q$. We show that if $pq<1/\gamma$ then $\partial g/\partial p>0$. We have
$$\partial L_\kappa/\partial p=h(q)+qh'(p)+\kappa\geq h(q)+qh'(p).$$
If $p< 1/2$ then this is positive. If $p\geq 1/2$, then
$$\partial L_\kappa/\partial p\geq h\left(\frac{1}{2.25p}\right)+\frac{h'(p)}{2.25p}$$
which is positive on $[1/2,2/3]$.
\end{proof}
\begin{lemma} \label{Lagrangesuff}
	Suppose $\kappa\geq 0$, $\gamma\geq 2.25$, $x\geq 2$ and assume that for $0\leq p, q\leq 1$, $pq=1/\gamma$ the maximum of \\$L_\kappa(p,q)=f(p,q)+\kappa(p+q)$ is $\psi(\gamma, x, \kappa)$. Then $\varphi(\gamma,x)\leq \psi(\gamma,x,\kappa)-2\kappa(1-1/x)$.
\end{lemma}
\begin{proof}
	If $(p_i)_{i=1}^n$, $(q_i)_{i=1}^n$ satisfy the constraints of ($\ref{optimisation}$), then
$$\frac{1}{n}\sum_{i=1}^{n} f(p_i,q_i)\leq \frac{1}{n}\sum_{i=1}^{n} {(f(p_i,q_i)+\kappa (p_i+q_i))}-\frac{1}{n}\kappa\cdot 2n(1-1/x).$$
Using Lemma \ref{whereistheoptimum} and our assumptions above, the result follows.
\end{proof}

\begin{lemma}\label{sigmaopt} Suppose $\kappa\geq 0$, $q=q(p)=1/(\gamma p)$, and $(p_0,q_0)$ satisfy $p_0q_0=1/\gamma$, $0\leq p_0,q_0\leq 1$ and
$$\kappa=\frac{p_0q_0}{\log{2}}\frac{g(p_0)-g(q_0)}{q_0-p_0}$$
where $g(x)=\frac{\log(1-x)}{x}$. Then $L_\kappa(p,q(p))$ is maximal at $(p_0,q_0)$.
\end{lemma}
\begin{proof}
We may assume $q>p$. As $dq/dp=-q/p$, we have (see \cite{holzmankorner} for more details)
$$\frac{d}{dp}\bigg[f(p,q(p))+\kappa (p+q(p))\bigg]=q\left[\frac{1}{p}\log_2{(1-p)}-\frac{1}{q}\log_2{(1-q)}\right]+\kappa(1-q/p).$$
This has the same sign as
$$\frac{pq}{\log{2}}\frac{g(p)-g(q)}{q-p}-\kappa$$
where $g(x)=\frac{\log(1-x)}{x}$. As $pq$ is constant, it suffices to show that in the range $\frac{1}{\gamma}\leq p<\frac{1}{\sqrt{\gamma}}$, the function
$$\sigma(p)=\frac{g(p)-g(q(p))}{q(p)-p}$$
is strictly decreasing. We have
\begin{align*}
\sigma'(p)=\frac{(g'(p)-g'(q)(-q/p))(q-p)-(g(p)-g(q))(-q/p-1)}{(q-p)^2}.
\end{align*}
Since $g'(x)=-\frac{1}{x(1-x)}-g(x)/x$, we obtain
\begin{align*}
p(q-p)^2\sigma'(p)&=(q-p)(pg'(p)+qg'(q))+(p+q)(g(p)-g(q))\\
&=(q-p)\left(-\frac{1}{1-p}-g(p)-\frac{1}{1-q}-g(q)\right)+(p+q)(g(p)-g(q))=\\
&=-(q-p)\left(\frac{1}{1-p}+\frac{1}{1-q}\right)+2pg(p)-2qg(q).
\end{align*}
Using the substitutions $1-p=x$, $1-q=y$, $a=x/y>1$, we get
\begin{align*}
p(q-p)^2\sigma'(p)&=-(x-y)\left(\frac{1}{x}+\frac{1}{y}\right)+2(\log{x}-\log{y})\\
&=-a+\frac{1}{a}+2\log{a}.
\end{align*}
But this is negative for $a>1$, since it is $0$ at $a=1$ and its derivative is
$$-1-\frac{1}{a^2}+\frac{2}{a}=-\frac{(1-a)^2}{a^2}.$$
So $\sigma$ is strictly decreasing.
\end{proof}

\begin{lemma}\label{efficientoptimization} Let $\gamma\geq 2.25$, $x\geq 2$ and let $(p_0,q_0)$ satisfy $p_0+q_0=2(1-1/x)$ and $p_0q_0=1/\gamma$.\\
If $0\leq p_0, q_0\leq 1$, $p_0\not=q_0$, then $\varphi(\gamma,x)=f(p_0,q_0)$.
\end{lemma}
\begin{proof}
	Choose $$\kappa=\frac{p_0q_0}{\log{2}}\frac{g(p_0)-g(q_0)}{q_0-p_0}$$
(this is positive, since $g'(x)<0$, see \cite{holzmankorner}.) By Lemma \ref{sigmaopt}, $\psi(\gamma, x, \kappa)=L_\kappa(p_0,q_0)$. By Lemma \ref{Lagrangesuff}, $\varphi(\gamma, x)\leq f(p_0,q_0)$. Equality can be achieved by choosing $n=2$, $p_1=q_2=p_0$, $p_2=q_1=q_0$.
\end{proof}

\subsection*{Implementation using a computer} Given $2=\lambda_0<\lambda_1<...<\lambda_N$ and $\rho_0=2$, the program iteratively looks for $\rho_{i+1}$ such that $r_1=\rho_i$, $r_2=\rho_{i+1}$, $\lambda=\lambda_i$, $\mu=\lambda_{i+1}$ satisfy (\ref{inductiveineq}). To make sure that rounding errors can be ignored, we require that the inequality holds with a difference of at least $\delta=10^{-8}$. We look for a minimal $\rho_{i+1}$ with these constraints. While searching for appropriate $\gamma$, the upper bounds we use for $\varphi(\gamma,x)$ are as follows.
\begin{enumerate}
	\item If Lemma \ref{efficientoptimization} can be used, we use it.\vspace{-6pt}
	\item If for $p,q\in[0,1]$, we have $pq\leq \gamma \implies p+q<2(1-1/x)$, then $\varphi(\gamma,x)=-\infty$\vspace{-6pt}
	\item If none of the above holds, we choose some $p_0$ and $q_0$ and apply Lemmas \ref{sigmaopt} and \ref{Lagrangesuff}.
\end{enumerate}
We note however, that for the final values of $\gamma$ we obtain,  we only need to use Case 1 (and Case 2 for $n/k$ close to $3.6$): Case 3 is only used to ease the search.
\section*{Acknowledgements}
I would like to thank Imre Leader for suggesting this problem for investigation and for useful discussions throughout the research, which was supported by Trinity College, Cambridge.
\bibliography{Bibliography}
\bibliographystyle{abbrv}

\end{document}